\documentclass[11pt]{article}%
\usepackage{eurosym}
\usepackage{amssymb}
\usepackage{amsmath}
\usepackage{amsfonts}
\usepackage{color}
\usepackage{fancyhdr}
\usepackage{graphicx}
\usepackage{hyperref}%
\providecommand{\U}[1]{\protect\rule{.1in}{.1in}}
\newtheorem{proposition}{Proposition}[section]
\newtheorem{theorem}[proposition]{Theorem}
\newtheorem{corollary}[proposition]{Corollary}
\newtheorem{lemma}[proposition]{Lemma}

\newtheorem{definition}[proposition]{Definition}
\newtheorem{remark}[proposition]{Remark}

\newtheorem{condition}[proposition]{Condition}

\numberwithin{equation}{section}
\numberwithin{proposition}{section}
\newenvironment{proof}[1][Proof]{\noindent\textbf{#1.} }{\ \rule{0.5em}{0.5em}}

\def\be{\begin{equation}}
\def\ee{\end{equation}}

\begin{document}

\title{Limits of relative entropies associated with weakly interacting particle systems}
\author{Amarjit Budhiraja\thanks{Research supported in part by National Science
Foundation(DMS-1305120) and the Army Research Office (W911NF-10-1-0158,
W911NF- 14-1-0331)}, Paul Dupuis\thanks{Research supported in part by the Army
Research Office (W911NF-12-1-0222).}, Markus Fischer and Kavita Ramanan\thanks{Research supported in part by the Army Research Office (W911NF-12-1-0222)
and the National Science Foundation (NSF CMMI-1234100 and NSF DMS-1407504)}}
\maketitle

\begin{abstract}
The limits of scaled relative entropies between probability distributions
associated with $N$-particle weakly interacting Markov processes are
considered. The convergence of such scaled relative entropies is established
in various settings. The analysis is motivated by the role relative entropy
plays as a Lyapunov function for the (linear) Kolmogorov forward equation
associated with an ergodic Markov process, and Lyapunov function properties of
these scaling limits with respect to nonlinear finite-state Markov processes
are studied in the companion paper \cite{BDFR-Examples}.

\end{abstract}


\noindent\emph{2010 Mathematics Subject Classification. } Primary: 60K35,
93D30, 34D20; Secondary: 60F10, 60K25. \newline

\noindent\emph{Key Words and Phrases. } Nonlinear Markov processes, weakly
interacting particle systems, interacting Markov chains, mean field limit,
stability, metastability, Lyapunov functions, relative entropy, large deviations.

\section{Introduction}

\label{sec:intro} We consider a collection of $N$ weakly interacting
particles, in which each particle evolves as a continuous time pure jump
c\`{a}dl\`{a}g stochastic process taking values in a finite state space
$\mathcal{X}=\{1,\ldots,d\}$. The evolution of this collection of particles is
described by an $N$-dimensional time-homogeneous Markov process
$\boldsymbol{X}^{N}=\{X^{i,N}\}_{i=1,\ldots,N}$, where for $t\geq0$,
$X^{i,N}(t)$ represents the state of the $i$th particle at time $t$. The jump
intensity of any given particle depends on the configuration of other
particles only through the empirical measure
\begin{equation}
\mu^{N}(t)\doteq\frac{1}{N}\sum_{i=1}^{N}\delta_{X^{i,N}(t)},\quad t\in
\lbrack0,\infty), \label{def-mun}%
\end{equation}
where $\delta_{a}$ is the Dirac measure at $a$. Consequently, a typical
particle's effect on the dynamics of the given particle is of order $1/N$. For
this reason the interaction is referred to as a \textquotedblleft weak
interaction.\textquotedblright

Note that $\mu^{N}(t)$ is a random variable with values in the space
${\mathcal{P}}_{N}({\mathcal{X}})$ $\doteq{\mathcal{P}}({\mathcal{X}}%
)\cap\frac{1}{N}\mathbb{Z}^{d}$, where ${\mathcal{P}}({\mathcal{X}})$ is the
space of probability measures on ${\mathcal{X}}$, equipped with the usual
topology of weak convergence. In the setting considered here, at most one
particle will jump, i.e., change state, at a given time, and the jump
intensities of any given particle depend only on its own state and the state
of the empirical measure at that time. In addition, the jump intensities of
all particles will have the same functional form. Thus, if the initial
particle distribution of $\boldsymbol{X}^{N}(0)=\{X^{i,N}(0)\}_{i=1,\ldots,N}$
is exchangeable, then at any time $t>0$, $\boldsymbol{X}^{N}(t)=\{X^{i,N}%
(t)\}_{i=1,\ldots,N}$ is also exchangeable.

Such mean field weakly interacting processes arise in a variety of
applications ranging from physics and biology to social networks and
telecommunications, and have been studied in many works (see, e.g.,
\cite{AntFriRobTib08,BorMcdPro08,BorMcdPro12,GomGraLeB12,GraRob10,McK66,Szn91}%
). The majority of this research has focused on establishing so-called
\textquotedblleft propagation-of-chaos\textquotedblright\ results (see, e.g.,
\cite{HuaMalCai06, Gra00,GomGraLeB12, KolBook11,Kol12,McK66, Oes84, Szn91}).
Roughly speaking, such a result states that on any fixed time interval
$[0,T]$, the particles become asymptotically independent as $N\rightarrow
\infty$, and that for each fixed $t$ the distribution of a typical particle
converges to a probability measure $p(t)$, which coincides with the limit in
probability of the sequence of empirical measures $\{\mu^{N}(t)\}_{N\in
\mathbb{N}}$ as $N\rightarrow\infty$. Under suitable conditions, the function
$t\mapsto p(t)$ can be characterized as the unique solution of a nonlinear
differential equation on ${\mathcal{P}}({\mathcal{X}})$ of the form
\begin{equation}
\frac{d}{dt}p(t)=p(t)\Gamma(p(t)), \label{EqLimitKolmogorov}%
\end{equation}
where for each $p\in{\mathcal{P}}({\mathcal{X}})$, $\Gamma(p)$ is a rate
matrix for a Markov chain on ${\mathcal{X}}$. This differential equation
admits an interpretation as the forward equation of a \textquotedblleft
nonlinear\textquotedblright\ jump Markov process that represents the evolution
of the typical particle. In the context of weakly interacting diffusions, this
limit equation is also referred to as the McKean-Vlasov limit.

Other work on mean field weakly interacting processes has established central
limit theorems \cite{Tan, Sznit,Mel,KuXi2,BuKiSa} or sample path large
deviations of the sequence $\{\mu^{N}\}$ \cite{DawGar,BDF,DupRamWu12}. All of
these results are concerned with the behavior of the $N$-particle system over
a finite time interval $[0,T]$.

\subsection{Discussion of main results}

\label{sec:discmainres} An important but difficult issue in the study of
nonlinear Markov processes is stability. Here, what is meant is the stability
of the ${\mathcal{P}}({\mathcal{X}})$-valued deterministic dynamical system
$\{p(t)\}_{t\geq0}$. For example, one can ask if there is a unique, globally
attracting fixed point for the ordinary differential equation (ODE)
\eqref{EqLimitKolmogorov}. When this is not the case, all the usual questions
regarding stability of deterministic systems, such as existence of multiple
fixed points, their local stability properties, etc., arise here as well. One
is also interested in the connection between these sorts of stability
properties of the limit model and related stability and metastability (in the
sense of small noise stochastic systems) questions for the prelimit model.

There are several features which make stability analysis particularly
difficult for these models. One is that the state space of the system, being
the set of probability measures on $\mathcal{X}$, is not a linear space
(although it is a closed, convex subset of a Euclidean space). A standard
approach to the study of stability is through construction of suitable
Lyapunov functions. Obvious first choices for Lyapunov functions, such as
quadratic functions, are not naturally suited to such state spaces. Related to
the structure of the state space is the fact that the dynamics, linearized at
any point in the state space, always have a zero eigenvalue, which also
complicates the stability analysis.

The purpose of the present paper and the companion paper \cite{BDFR-Examples}
is to introduce and develop a systematic approach to the construction of
Lyapunov functions for nonlinear Markov processes. The starting point is the
observation that given any ergodic Markov process, the mapping $q\mapsto
R(q\Vert\pi)$, where $R$ is relative entropy and $\pi$ is the stationary
distribution, in a certain sense always defines a Lyapunov function for the
distribution of the Markov process \cite{Spi71}. We discuss this point in some
detail in Section \ref{sec:descent}. For an ergodic Markov process the
dynamical system describing the evolution of the law of the process (i.e., the
associated Kolmogorov's forward equation) is a linear ODE with a unique fixed
point. In contrast, for a nonlinear Markov process the corresponding ODE
\eqref{EqLimitKolmogorov} can have multiple fixed points which may or may not
be locally stable, and this is possible even when the jump rates given by the
off diagonal elements of $\Gamma(p)$ are bounded away from $0$ uniformly in
$p$. Furthermore, as is explained in Section \ref{sec:descent}, due to the
nonlinearity of $\Gamma(\cdot)$ relative entropy typically cannot be used
directly as a Lyapunov function for \eqref{EqLimitKolmogorov}.

The approach we take for nonlinear Markov processes is to lift the problem to
the level of the pre-limit $N$-particle processes that describe a linear
Markov process. Under mild conditions the $N$-particle process will be
ergodic, and thus relative entropy can be used to define a Lyapunov function
for the joint distribution of these $N$ particles. The scaling properties of
relative entropy and convergence properties of the weakly interacting system
then suggest that the limit of suitably normalized relative entropies for the
$N$-particle system, assuming it exists, is a natural candidate Lyapunov
function for the corresponding nonlinear Markov process. Specifically,
denoting the unique invariant measure of the $N$-particle Markov process
$\boldsymbol{X}^{N}$ by $\boldsymbol{\pi}_{N}\in\mathcal{P}(\mathcal{X}^{N})$,
the function $F:\mathcal{P}(\mathcal{X})\rightarrow\mathbb{R}$ defined by
\begin{equation}
F(q)=\lim_{N\rightarrow\infty}\tilde{F}_{N}(q)\doteq\lim_{N\rightarrow\infty
}\frac{1}{N}R\left(  \left.  \otimes^{N}q\right\Vert \boldsymbol{\pi}%
_{N}\right)  ,\;q\in\mathcal{P}(\mathcal{X}) \label{eq:firstlyap}%
\end{equation}
is a natural candidate for a Lyapunov function. The aim of this paper is the
calculation of quantities of the form (\ref{eq:firstlyap}). In the companion
paper \cite{BDFR-Examples} we will use these results to construct Lyapunov
functions for various particular systems.

Of course for this approach to work, we need the limit on the right side in
(\ref{eq:firstlyap}) to exist and to be computable. In Section
\ref{sec:gibbstype} we introduce a family of nonlinear Markov processes that
arises as the large particle limit of systems of Gibbs type. For this family,
the invariant distribution of the corresponding $N$-particle system takes an
explicit form and we show that the right side of \eqref{eq:firstlyap} has a
closed form expression. In Section 4 of \cite{BDFR-Examples} we show that this
limiting function is indeed a Lyapunov function for the corresponding
nonlinear dynamical system \eqref{EqLimitKolmogorov}.

The class of models just mentioned demonstrates that the approach for
constructing Lyapunov functions by studying scaling limits of the relative
entropies associated with the corresponding $N$-particle Markov processes has
merit. However, for typical nonlinear systems as in \eqref{EqLimitKolmogorov},
one does not have an explicit form for the stationary distribution of the
associated $N$-particle system, and thus the approach of computing limits of
$\tilde{F}_{N}$ as in (\ref{eq:firstlyap}) becomes infeasible. An alternative
is to consider the limits of
\begin{equation}
F_{t}^{N}(q)\doteq\frac{1}{N}R(\otimes^{N}q\Vert\boldsymbol{p}^{N}(t)),
\label{eq:fnq}%
\end{equation}
where $\boldsymbol{p}^{N}(t)$ is the (exchangeable) probability distribution
of $\boldsymbol{X}^{N}(t)$ with some exchangeable initial distribution
$\boldsymbol{p}^{N}(0)$ on $\mathcal{X}^{N}$. Formally taking the limit of
$F_{t}^{N}$, first as $t\rightarrow\infty$ and then as $N\rightarrow\infty$,
we arrive at the function $F$ introduced in (\ref{eq:firstlyap}). Since as we
have noted this limit cannot in general be evaluated, one may instead attempt
to evaluate the limit in the reverse order, i.e., send $N\rightarrow\infty$
first, followed by $t\rightarrow\infty$.

A basic question one then asks is whether the limit $\lim_{N\rightarrow\infty
}F_{t}^{N}(q)$ takes a useful form. In Section \ref{sec:asymgenexc} we will
answer this question in a rather general setting. Specifically, we show that
under suitable conditions the limit of $\frac{1}{N}R(\otimes^{N}q\Vert Q^{N})$
as $N\rightarrow\infty$ exists for every $q\in\mathcal{P}(\mathcal{X})$ and
exchangeable sequence $\{Q^{N}\}_{N\in\mathbb{N}},Q^{N}\in\mathcal{P}%
(\mathcal{X}^{N})$. The main condition needed is that the collection of
empirical measures of $N$ random variables with joint distribution $Q^{N}$
satisfies a locally uniform large deviation principle (LDP) on $\mathcal{P}%
(\mathcal{X})$. We show in this case that the limit of $\frac{1}{N}%
R(\otimes^{N}q\Vert Q^{N})$ is given by $J(q)$, where $J$ is the rate function
associated with the LDP. Applying this result to $Q^{N}=\boldsymbol{p}^{N}%
(t)$, we then identify the limit as $N\rightarrow\infty$ of $F_{t}^{N}(q)$ as
$J_{t}(q)$, where $J_{t}$ is the large deviations rate function for the
collection of $\mathcal{P}(\mathcal{X})$-valued random variables $\{\mu
^{N}(t)\}_{N\in\mathbb{N}}$. In the companion paper we will show that the
limit of $J_{t}(q)$ as $t\rightarrow\infty$ yields a Lyapunov function for
\eqref{EqLimitKolmogorov} for interesting models, including a class we call
\textquotedblleft locally Gibbs,\textquotedblright\ which generalizes those
obtained as limits of $N$-particle Gibbs models.

\subsection{Outline of the paper and common notation}

The paper is organized as follows. In Section \ref{SectInterMarkov} we
describe the interacting particle system model and the ODE that characterizes
its scaling limit. Section \ref{sec:descent} recalls the descent property of
relative entropy for (linear) Markov processes. Section \ref{sec:gibbstype}
studies systems of Gibbs type and shows how a Lyapunov function can be
obtained by evaluating limits of $\tilde{F}_{N}(q)$ as $N\rightarrow\infty$.
Next, in Section \ref{sec:genweak} we consider models more general than Gibbs
systems. In Section \ref{sec:asymgenexc}, we carry out an asymptotic analysis
of $\frac{1}{N}R(\otimes^{N}q\Vert Q^{N})$ as $N\rightarrow\infty$ for an
exchangeable sequence $\{Q^{N}\}_{N\in\mathbb{N}}$. The results of Section
\ref{sec:asymgenexc} are then used in Section \ref{sec:quasandmeta} to
evaluate $\lim_{N\rightarrow\infty}F_{t}^{N}(q)$. Section
\ref{sec:quasandmeta} also contains remarks on relations between the
constructed Lyapunov functions and the Freidlin-Wentzell quasipotential and
metastability issues for the underlying $N$-particle Markov process.

The following notation will be used. Given any Polish space $E$,
$D([0,\infty):E)$ denotes the space of $E$-valued right continuous functions
on $[0,\infty)$ with finite left limits on $(0,\infty)$, equipped with the
usual Skorohod topology. Weak convergence of a sequence $\{X_{n}\}$ of
$E$-valued random variables to a random variable $X$ is denoted by
$X_{n}\Rightarrow X$. The cardinality of a finite set $C$ is denoted by $|C|$.

\section{Background and Model Description}

\label{SectInterMarkov}

\subsection{Description of the $N$-particle system}

\label{subs-model}

In this section, we provide a precise description of the time-homogeneous
${\mathcal{X}}^{N}$-valued Markov process ${\boldsymbol{X}}^{N}=(X^{1,N}%
,\ldots,X^{N,N})$ that describes the evolution of the $N$-particle system. We
assume that at most one particle can change state at any given time. Models
for which more than one particle can change state simultaneously are also
common \cite{AntFriRobTib08, GibHunKel90, DupRamWu12}. However, under broad
conditions the limit \eqref{EqLimitKolmogorov} for such models also has an
interpretation as the forward equation of a model in which only one particle
can change state at any time \cite{DupRamWu12}, and so for purposes of
stability analysis of \eqref{EqLimitKolmogorov} this assumption is not much of
a restriction.

Recall that ${\mathcal{X}}$ is the finite set $\{1,\ldots,d\}$. The
transitions of ${\boldsymbol{X}}^{N}$ are determined by a family of matrices
$\{\Gamma^{N}(r)\}_{r\in{\mathcal{P}}({\mathcal{X}})}$, where for each
$r\in{\mathcal{P}}({\mathcal{X}})$, $\Gamma^{N}(r)=\{\Gamma_{x,y}%
^{N}(r),x,y\in{\mathcal{X}}\}$ is a transition rate matrix of a continuous
time Markov chain on ${\mathcal{X}}$. For $y\neq x$, $\Gamma_{xy}^{N}(r)\geq0$
represents the rate at which a single particle transitions from state $x$ to
state $y$ when the empirical measure has value $r$. More precisely, the
transition mechanism of $\boldsymbol{X}^{N}$ is as follows. Given
$\boldsymbol{X}^{N}(t)=\boldsymbol{x}\in\mathcal{X}^{N}$, an index
$i\in\left\{  1,\ldots,N\right\}  $ and $y\neq x_{i}$, the jump rate at time
$t$ for the transition
\[
\left(  x_{1},\ldots,x_{i-1},x_{i},x_{i+1},\ldots,x_{N}\right)  \rightarrow
\left(  x_{1},\ldots,x_{i-1},y,x_{i+1},\ldots,x_{N}\right)
\]
is $\Gamma_{x_{i}y}^{N}(r^{N}(\boldsymbol{x}))$, where $r^{N}(\boldsymbol{x}%
)\in{\mathcal{P}}_{N}({\mathcal{X}})$ is the empirical measure of the vector
$\boldsymbol{x}\in{\mathcal{X}}^{N}$, which is given explicitly by
\begin{equation}
r_{y}^{N}(\boldsymbol{x})\doteq\frac{1}{N}\sum_{i=1}^{N}1_{\{x_{i}=y\}},\qquad
y\in{\mathcal{X}}. \label{def-rn}%
\end{equation}
Moreover, the jump rates for transitions of any other type are zero. Note that
$r_{\cdot}^{N}(\boldsymbol{X}^{N}(t))$ equals the empirical measure $\mu
^{N}(t)(\cdot)$, defined in (\ref{def-mun}).

The description in the last paragraph completely specifies the infinitesimal
generator or rate matrix of the $\mathcal{X}^{N}$-valued Markov process
${\boldsymbol{X}}^{N}$, which we will denote throughout by $\boldsymbol{\Psi
}^{N}$. Note that the sample paths of ${\boldsymbol{X}}^{N}$ lie in
$D([0,\infty):\mathcal{X}^{N})$, where ${\mathcal{X}}^{N}$ is endowed with the
discrete topology. The generator $\boldsymbol{\Psi}^{N},$ together with a
collection of $\mathcal{X}$-valued random variables $\{X^{i,N}%
(0)\}_{i=1,\ldots,N}$ whose distribution we take to be exchangeable,
determines the law of $\boldsymbol{X}^{N}$.

\subsection{The jump Markov process for the empirical measure}

\label{subs-jmp}

As noted in Section \ref{sec:intro}, exchangeability of the initial random
vector
\[
\{X^{i,N}(0),i=1,\ldots,N\}
\]
implies that the processes $\{X^{i,N}\}_{i=1,\ldots,N}$ are also exchangeable.
From this, it follows that the empirical measure process $\mu^{N}=\{\mu
^{N}(t)\}_{t\geq0}$ is a Markov chain taking values in ${\mathcal{P}}%
_{N}({\mathcal{X}})$. We now describe the evolution of this measure-valued
Markov chain. For $x\in{\mathcal{X}}$, let $e_{x}$ denotes the unit coordinate
vector in the $x$-direction in $\mathbb{R}^{d}$. Since almost surely at most
one particle can change state at any given time, the possible jumps of $\mu^N$ are of the
form $v/N,v\in{\mathcal{V}}$, where
\begin{equation}
{\mathcal{V}}\doteq\left\{  e_{y}-e_{x}:x,y\in{\mathcal{X}}:x\neq y\right\}  .
\label{def-jumpset}%
\end{equation}
Moreover, if $\mu^{N}(t)=r$ for some $r\in{\mathcal{P}}_{N}({\mathcal{X}})$,
then at time $t$, $Nr_{x}$ of the particles are in the state $x$. Therefore,
the rate of the particular transition $r\rightarrow r+(e_{y}-e_{x})/N$ is
$Nr_{x}\Gamma_{xy}^{N}(r)$. Consequently the generator ${\mathcal{L}}^{N}$ of
the jump Markov process $\mu^{N}$ is given by
\begin{equation}
{\mathcal{L}}^{N}f(r)=\sum_{x,y\in{\mathcal{X}}:x\neq y}Nr_{x}\Gamma_{xy}%
^{N}(r)\left[  f\left(  r+\frac{1}{N}(e_{y}-e_{x})\right)  -f(r)\right]
\label{eq:gen}%
\end{equation}
for real-valued functions $f$ on ${\mathcal{P}}_{N}({\mathcal{X}})$.

\subsection{Law of large numbers limit}

\label{subs-lln}

We now characterize the law of large numbers limit of the sequence $\{\mu
^{N}\}_{N\in\mathbb{N}}$. It will be convenient to identify ${\mathcal{P}%
}({\mathcal{X}})$ with the $(d-1)$-dimensional simplex ${\mathcal{S}}$ in
$\mathbb{R}^{d}$, given by
\begin{equation}
{\mathcal{S}}\doteq\left\{  x\in\mathbb{R}^{d}:\sum_{i=1}^{d}x_{i}=1,x_{i}%
\geq0,i=1,\ldots,d\right\}  , \label{simplex}%
\end{equation}
and identify ${\mathcal{P}}_{N}({\mathcal{X}})$ with ${\mathcal{S}}_{N}%
\doteq{\mathcal{S}}\cap\frac{1}{N}\mathbb{Z}^{d}$. We use ${\mathcal{P}%
}({\mathcal{X}})$ and ${\mathcal{S}}$ (likewise, ${\mathcal{P}}_{N}%
({\mathcal{X}})$ and ${\mathcal{S}}_{N}$) interchangeably. We endow
${\mathcal{S}}$ with the usual Euclidean topology and note that this
corresponds to ${\mathcal{P}}({\mathcal{X}})$ endowed with the topology of
weak convergence. We also let ${\mathcal{S}}^{\circ}$ denote the relative
interior of ${\mathcal{S}}$.

\begin{condition}
\label{ass-llnbasic} For every pair $x,y\in{\mathcal{X}}$, $x\neq y$, there
exists a Lipschitz continuous function $\Gamma_{xy}:{\mathcal{S}}%
\rightarrow\lbrack0,\infty)$ such that $\Gamma_{xy}^{N}\rightarrow\Gamma_{xy}$
uniformly on ${\mathcal{S}}$.
\end{condition}

We will find it convenient to define $\Gamma_{xx}(r)\doteq-\sum_{y\in
{\mathcal{X}},y\neq x}\Gamma_{yx}(r)$, so that $\Gamma(r)$ can be viewed as a
$d\times d$ transition rate matrix of a jump Markov process on ${\mathcal{X}}$.

Laws of large numbers for the empirical measures of interacting processes can
be efficiently established using a martingale problem formulation, see for
instance \cite{Oes84}. Since $\mathcal{X}$ is finite, in the present situation
we can rely on a classical convergence theorem for pure jump Markov processes
with state space contained in a Euclidean space.

\begin{theorem}
\label{ThLLN} Suppose that Condition \ref{ass-llnbasic} holds, and assume that
$\mu^{N}(0)$ converges in probability to $q\in\mathcal{P}(\mathcal{X})$ as $N$
tends to infinity. Then $\{\mu^{N}(\cdot)\}_{N\in\mathbb{N}}$ converges
uniformly on compact time intervals in probability to $p(\cdot)$, where
$p(\cdot)$ is the unique solution to \eqref{EqLimitKolmogorov} with $p(0)=q$.
\end{theorem}

\begin{proof}
The assertion follows from Theorem~2.11 in \cite{Kur70}. In the notation of
that work, $E=\mathcal{P}(\mathcal{X})$, $E_{N}=\mathcal{P}_{N}(\mathcal{X})$,
$N\in\mathbb{N}$,
\begin{align*}
&  F_{N}(p)=\sum_{x,y\in\mathcal{X}}N\cdot p_{x}\left(  \tfrac{1}{N}%
e_{y}-\tfrac{1}{N}e_{x}\right)  \Gamma_{x,y}^{N}(p), &  &  p\in E_{N},\\
&  F(p)=\sum_{x,y\in\mathcal{X}}p_{x}(e_{y}-e_{x})\Gamma_{x,y}(p), &  &  p\in
E,
\end{align*}
and we recall $e_{x}$ is the unit vector in $\mathbb{R}^{d}$ with component
$x$ equal to $1$. Note that if $f$ is the identity function $f(\tilde
{p})\doteq\tilde{p}\in\mathbb{R}^{d}$, then $F_{N}(p)=\mathcal{L}^{N}f(p)$,
$p\in\mathcal{P}_{N}(\mathcal{X})$, where $\mathcal{L}^{N}$ is the generator
given in (\ref{eq:gen}). Moreover, the $z$-th component of the $d$-dimensional
vector $F(p)$ is equal to $\sum_{x: x\neq z}p_{x}\Gamma_{x,z}(p)-\sum_{y:
y\neq z}p_{z}\Gamma_{z,y}(p)$, which in turn is equal to $\sum_{x}p_{x}%
\Gamma_{x,z}(p)$, the $z$-th component of the row vector $p\Gamma(p)$. The ODE
$\frac{d}{dt}p(t)=F(p(t))$ is therefore the same as \eqref{EqLimitKolmogorov}.
Since $F$ is Lipschitz continuous by Condition \ref{ass-llnbasic}, this ODE
has a unique solution.  The proof is now immediate from Theorem~2.11 in \cite{Kur70}.
\end{proof}

\medskip The solution to (\ref{EqLimitKolmogorov}) has a stochastic
representation. Given a probability measure $q(0)\in$ ${\mathcal{P}%
}(\mathcal{X})$, one can construct a process $X$ with sample paths in
$D([0,T]:\mathcal{X})$ such that for all functions $f:\mathcal{X}%
\rightarrow\mathbb{R}$,
\[
f(X(t))-f(X(0))-\int_{0}^{t}\sum_{y\in\mathcal{X}}\Gamma_{X(s)y}(q(s))f(y)ds
\]
is a martingale, where $q(t)$ denotes the probability distribution of $X(t)$,
$t\geq0$. Furthermore, $X$ is unique in law. Note that the rate matrix of
$X(t)$ is time inhomogeneous and equal to $\Gamma(q(t))$, with $q_{x}(t)$
$=P\left\{  X(t)=x\right\}  $. Because the evolution of $X$ at time $t$
depends on the distribution of $X(t)$, this process is called a nonlinear
Markov process. Note that $q(t)$ also solves (\ref{EqLimitKolmogorov}), and so
if $q(0)=p(0)$, by uniqueness $p_{x}(t)=P\left\{  X(t)=x\right\}  $. One can
show that, under the conditions of Theorem \ref{ThLLN}, $X(\cdot)$ is the
limit in distribution of $X^{i,N}(\cdot)$ for any fixed $i$, as $N\rightarrow
\infty$ (see Proposition 2.2 of \cite{Szn91}).

A fundamental property of interacting systems that will play a role in the
discussion below is propagation of chaos; see \cite{gottlieb98} for an
exposition and characterization. Propagation of chaos means that the first $k$
components of the $N$-particle system over any finite time interval will be
asymptotically independent and identically distributed (i.i.d.) as $N$ tends
to infinity, whenever the initial distributions of all components are
asymptotically i.i.d. In the present context, propagation of chaos for the
family $(\boldsymbol{X}^{N})_{N\in\mathbb{N}}$ (or $\{\Psi^{N}\}_{N\in
\mathbb{N}}$) means the following. For $t\geq0$ denote the probability
distribution of $(X^{1,N}(t),\ldots,X^{k,N}(t))$ by $\boldsymbol{p}^{N,k}(t)$.
If $q\in\mathcal{P}(\mathcal{X})$ and if for all $k\in\mathbb{N}$
$\boldsymbol{p}^{N,k}(0)$ converges weakly to the product measure $\otimes
^{k}q$ as $N\rightarrow\infty$, then for all $k\in N$ and all $t\geq0$
$\boldsymbol{p}^{N,k}(t)$ converges weakly to $\otimes^{k}p(t)$, where
$p(\cdot)$ is the solution to \eqref{EqLimitKolmogorov} with $p(0)=q$. Instead
of a particular time $t$ a finite time interval may be considered. Under the
assumptions of Theorem~\ref{ThLLN}, propagation of chaos holds for the family
of $N$-particle systems determined by $\{\Psi^{N}\}_{N\in\mathbb{N}}$. See,
for instance, Theorem 4.1 in \cite{graham92}.

\section{Descent Property of Relative Entropy for Markov Processes}

\label{sec:descent}

We next discuss an important property of the usual (linear) Markov processes.
As noted in the introduction, various features of the deterministic system
(\ref{EqLimitKolmogorov}) make standard forms of Lyapunov functions that might
be considered unsuitable. Indeed, one of the most challenging problems in the
construction of Lyapunov functions for any system is the identification of
natural forms that reflect the particular features and structure of the system.

The ODE (\ref{EqLimitKolmogorov}) is naturally related to a flow of
probability measures, and for this reason one might consider constructions
based on relative entropy. It is known that for an ergodic linear Markov
process relative entropy serves as a Lyapunov function. Specifically, relative
entropy has a descent property along the solution of the forward equation. The
earliest proof in the setting of finite-state continuous-time Markov processes
the authors have been able to locate is \cite[pp.\,I-16-17]{Spi71}. Since
analogous arguments will be used elsewhere (see Section 2 of
\cite{BDFR-Examples}), we give the proof of this fact. Let $G=(G_{x,y}%
)_{x,y\in\mathcal{X}}$ be an irreducible rate matrix over the finite state
space $\mathcal{X}$, and denote its unique stationary distribution by $\pi$.
The forward equation for the family of Markov processes with rate matrix $G$
is the linear ODE
\begin{equation}
\frac{d}{dt}r(t)=r(t)G. \label{EqLinearKolmogorov}%
\end{equation}
Define $\ell:$ $[0,\infty)\rightarrow\lbrack0,\infty)$ by $\ell(z)\doteq
z\log z-z+1$. Recall that the relative entropy of $p\in\mathcal{P}%
(\mathcal{X})$ with respect to $q\in\mathcal{P}(\mathcal{X})$ is given by
\begin{equation}
R\left(  p\Vert q\right)  \doteq\sum_{x\in\mathcal{X}}p_{x}\log\left(
\frac{p_{x}}{q_{x}}\right)  =\sum_{x\in\mathcal{X}}q_{x}\ell\left(
\frac{p_{x}}{q_{x}}\right)  . \label{eq:relent}%
\end{equation}

\begin{lemma}
\label{LemmaREDescent} Let $p(\cdot)$, $q(\cdot)$ be solutions to
\eqref{EqLinearKolmogorov} with initial conditions $p(0),q(0)\in
\mathcal{P}(\mathcal{X})$. Then for all $t>0$,
\[
\frac{d}{dt}R\left(  p(t)\Vert q(t)\right)  =-\hspace{-1ex}\sum_{x,y\in
\mathcal{X}:x\neq y}\ell\left(  \frac{p_{y}(t)q_{x}(t)}{p_{x}(t)q_{y}%
(t)}\right)  p_{x}(t)\frac{q_{y}(t)}{q_{x}(t)}G_{y,x}\leq0.
\]
Moreover, $\frac{d}{dt}R\left(  p(t)\Vert q(t)\right)  =0$ if and only if
$p(t)=q(t)$.
\end{lemma}

\begin{proof}
\vspace{0pt} \noindent It is well known (and easy to check) that $\ell$ is
strictly convex on $[0,\infty)$, with $\ell(0)=1$ and $\ell(z)=0$ if and only
if $z=1$. Owing to the irreducibility of $G$, for $t>0$ $p(t)$ and $q(t)$ have
no zero components and hence are equivalent probability vectors. By
assumption, $p_{x}^{\prime}(t)\doteq\frac{d}{dt}p_{x}(t)=\sum_{y\in
\mathcal{X}}p_{y}(t)G_{y,x}$ for all $x\in\mathcal{X}$ and all $t\geq0$, and
similarly for $q(t)$. Thus for $t>0$
\begin{align*}
\lefteqn{\frac{d}{dt}R\left(  p(t)\Vert q(t)\right)  }\\
&  \quad=\frac{d}{dt}\sum_{x\in\mathcal{X}}p_{x}(t)\log\left(  \frac{p_{x}%
(t)}{q_{x}(t)}\right) \\
&  \quad=\sum_{x\in\mathcal{X}}p_{x}^{\prime}(t)\log\left(  \frac{p_{x}%
(t)}{q_{x}(t)}\right)  +\sum_{x\in\mathcal{X}}p_{x}^{\prime}(t)-\sum
_{x\in\mathcal{X}}p_{x}(t)\frac{q_{x}^{\prime}(t)}{q_{x}(t)}\\
&  \quad=\sum_{x,y\in\mathcal{X}}\left(  p_{y}(t)+p_{y}(t)\log\left(
\frac{p_{x}(t)}{q_{x}(t)}\right)  -p_{x}(t)\frac{q_{y}(t)}{q_{x}(t)}\right)
G_{y,x}\\
&  \quad\quad\mbox{}-\sum_{x,y\in\mathcal{X}}p_{y}(t)\log\left(  \frac
{p_{y}(t)}{q_{y}(t)}\right)  G_{y,x},
\end{align*}
where the last equality follows from the fact that, since $G$ is a rate
matrix, $\sum_{x\in\mathcal{X}}G_{y,x}=0$ for all $y\in\mathcal{X}$.
Rearranging terms we have
\begin{align*}
\lefteqn{\frac{d}{dt}R\left(  p(t)\Vert q(t)\right)  }\\
&  \quad=\sum_{x,y\in\mathcal{X}}\left(  p_{y}(t)-p_{y}(t)\log\left(
\frac{p_{y}(t)q_{x}(t)}{p_{x}(t)q_{y}(t)}\right)  -p_{x}(t)\frac{q_{y}%
(t)}{q_{x}(t)}\right)  G_{y,x}\\
&  \quad=\sum_{x,y\in\mathcal{X}}\left(  \frac{p_{y}(t)q_{x}(t)}{p_{x}%
(t)q_{y}(t)}-\frac{p_{y}(t)q_{x}(t)}{p_{x}(t)q_{y}(t)}\log\left(  \frac
{p_{y}(t)q_{x}(t)}{p_{x}(t)q_{y}(t)}\right)  -1\right)  p_{x}(t)\frac
{q_{y}(t)}{q_{x}(t)}G_{y,x}\\
&  \quad=\mbox{}-\hspace{-1ex}\sum_{x,y\in\mathcal{X}:x\neq y}\ell\left(
\frac{p_{y}(t)q_{x}(t)}{p_{x}(t)q_{y}(t)}\right)  p_{x}(t)\frac{q_{y}%
(t)}{q_{x}(t)}G_{y,x}.
\end{align*}
Recall that $\ell\geq0$, that for $t>0$ $q_{x}(t)>0$ and $p_{x}(t)>0$ for all
$x\in\mathcal{X}$, and that $G_{y,x}\geq0$ for all $x\neq y$. It follows that
$\frac{d}{dt}R(p(t)\Vert q(t))\leq0$.

It remains to show that $\frac{d}{dt}R(p(t)\Vert q(t))=0$ if and only if
$p(t)=q(t)$. We claim this follows from the fact that $\ell\geq0$ with
$\ell(z)=0$ if and only if $z=1$, and from the irreducibility of $G$. Indeed,
$p(t)=q(t)$ if and only if $\frac{p_{y}(t)q_{x}(t)}{p_{x}(t)q_{y}(t)}=1$ for
all $x,y\in\mathcal{X}$ with $x\neq y$. Thus $p(t)=q(t)$ implies $\frac{d}%
{dt}R(p(t)\Vert q(t))=0$. If $\frac{d}{dt}R(p(t)\Vert q(t))=0$ then
immediately $\frac{p_{y}(t)q_{x}(t)}{p_{x}(t)q_{y}(t)}=1$ for all
$x,y\in\mathcal{X}$ such that $G_{y,x}>0$. If $y$ does not directly
communicate with $x$ then, by irreducibility, there is a chain of directly
communicating states leading from $y$ to $x$, and using those states it
follows that $\frac{p_{y}(t)q_{x}(t)}{p_{x}(t)q_{y}(t)}=1$.
\end{proof}

\medskip If $q(0)=\pi$ then, by stationarity, $q(t)=\pi$ for all $t\geq0$.
Lemma~\ref{LemmaREDescent} then implies that the mapping
\begin{equation}
p\mapsto R\left(  p\Vert\pi\right)  \label{ExLyapunovLinStat}%
\end{equation}
is a local (and also global) Lyapunov function (cf. Definition 2.4 in
\cite{BDFR-Examples}) for the linear forward equation
\eqref{EqLinearKolmogorov} on any relatively open subset of $\mathcal{S}$ that
contains $\pi$.

This is, however, just one of many ways that relative entropy can be used to
define Lyapunov functions. For example, Lemma~\ref{LemmaREDescent} also
implies%
\begin{equation}
p\mapsto R\left(  \pi\Vert p\right)  \label{eq:eqrevrelent}%
\end{equation}
is a local and global Lyapunov function for \eqref{EqLinearKolmogorov}. Yet a
third can be constructed as follows. Let $T>0$ and consider the mapping
\begin{equation}
p\mapsto R\left(  p\Vert q^{p}(T)\right)  , \label{ExLyapunovLinTime}%
\end{equation}
where $q^{p}(\cdot)$ is the solution to \eqref{EqLinearKolmogorov} with
$q^{p}(0)=p$. Lemma~\ref{LemmaREDescent} also implies that the mapping given
by \eqref{ExLyapunovLinTime} is a Lyapunov function for
\eqref{EqLinearKolmogorov}. This is because $R\left(  p(t)\Vert q^{p(t)}%
(T)\right)  =R\left(  p(t)\Vert q(t)\right)  $, where $q(\cdot)$ is the
solution to \eqref{EqLinearKolmogorov} with $q(0)=p(T)$, thus
$q(t)=p(T+t)=q^{p(t)}(T)$. Note that \eqref{ExLyapunovLinStat} arises as the
limit of \eqref{ExLyapunovLinTime} as $T$ goes to infinity.

The proof of the descent property in Lemma \ref{LemmaREDescent} crucially uses
the fact that $p(\cdot)$ and $q(\cdot)$ satisfy a forward equation with
respect to the same fixed rate matrix, and therefore for general nonlinear
Markov processes one does not expect relative entropy to serve directly as a
Lyapunov function. However, one might conjecture this to be true if the
nonlinearity is in some sense weak, and a result of this type is presented in
the companion paper \cite{BDFR-Examples} (see Section 3 therein). For more
general settings our approach will be to consider functions such as those in
\eqref{ExLyapunovLinStat} and \eqref{ExLyapunovLinTime} associated with the
$N$-particle Markov processes and then take a suitable scaling limit as
$N\rightarrow\infty$. The issue is somewhat subtle, e.g., while this approach
is feasible with the form \eqref{ExLyapunovLinStat} it is not feasible when
the form \eqref{eq:eqrevrelent} is used, even though both define Lyapunov
functions in the linear case. For further discussion on this point we refer to
Remark \ref{rem:notrev}.

\section{Systems of Gibbs Type}

\label{sec:gibbstype}

In this section we evaluate the limit in \eqref{eq:firstlyap} for a family of
interacting $N$-particle systems with an explicit stationary distribution.
This limit is shown to be a Lyapunov function in \cite{BDFR-Examples}.
Section~\ref{SectGibbsSystems} introduces the class of weakly interacting
Markov processes and the corresponding nonlinear Markov processes. The
construction starts from the definition of the stationary distribution as a
Gibbs measure for the $N$-particle system. In Section~\ref{SectGibbsRELimit}
we derive candidate Lyapunov functions for the limit systems as limits of
relative entropy.

\subsection{The prelimit and limit systems}

\label{SectGibbsSystems}

Recall that $\mathcal{X}$ is a finite set with $d\geq2$ elements. Let
$K:\mathcal{X}\times\mathbb{R}^{d}\rightarrow\mathbb{R}$ be such that for each
$x\in\mathcal{X}$, $K(x,\cdot)$ is twice continuously differentiable. For
$(x,p)\in\mathcal{X}\times\mathbb{R}^{d}$, we often write $K(x,p)$ as
$K^{x}(p)$. Consider the probability measure $\boldsymbol{\pi}_{N}$ on
$\mathcal{X}^{N}$ defined by
\begin{equation}
\boldsymbol{\pi}_{N}(\boldsymbol{x})\doteq\frac{1}{Z_{N}}\exp\left(
-U_{N}(\boldsymbol{x})\right)  ,\;\boldsymbol{x}\in\mathcal{X}^{N},
\label{eq:gibbsmzr}%
\end{equation}
where $Z_{N}$ is the normalization constant,
\begin{equation}
U_{N}(\boldsymbol{x})\doteq\sum_{i=1}^{N}K(x_{i},r^{N}(\boldsymbol{x}%
)),\;\boldsymbol{x}=(x_{1},\ldots x_{N})\in\mathcal{X}^{N}, \label{eq:norm}%
\end{equation}
and $r^{N}(\boldsymbol{x})$ is the empirical measure of $\boldsymbol{x}$ and
was defined in \eqref{def-rn} (recall we identify an element of $\mathcal{P}%
(\mathcal{X})$ with a vector in $\mathcal{S}$).

A particular example of $K$ that has been extensively studied is given by
\begin{equation}
K(x,p)\doteq V(x)+\beta\sum_{y\in\mathcal{X}}W(x,y)p_{y},\;(x,p)\in
\mathcal{X}\times\mathbb{R}^{d}, \label{eq:affinegibbs}%
\end{equation}
where $V:\mathcal{X}\rightarrow\mathbb{R}$ is referred to as the
\emph{environment potential}, $W:\mathcal{X}\times\mathcal{X}\rightarrow
\mathbb{R}$ the \emph{interaction potential}, and $\beta>0$ the
\emph{interaction parameter}. In this case $U_{N}$, referred to as the
$N$-particle energy function, takes the form
\[
U_{N}(\boldsymbol{x})=\sum_{i=1}^{N}V(x_{i})+\frac{\beta}{N}\sum_{i=1}^{N}%
\sum_{j=1}^{N}W(x_{i},x_{j}).
\]

There are standard methods for identifying $\mathcal{X}^{N}$-valued Markov
processes for which $\boldsymbol{\pi}_{N}$ is the stationary distribution. The
resulting rate matrices are often called Glauber dynamics; see, for instance,
\cite{stroock05} or \cite{martinelli99}. To be precise, we seek an
$\mathcal{X}^{N}$-valued Markov process which has the structure of a weakly
interacting $N$-particle system and is reversible with respect to
$\boldsymbol{\pi}_{N}$.

Let $(\alpha(x,y))_{x,y\in\mathcal{X}}$ be an irreducible and symmetric matrix
with diagonal entries equal to zero and off-diagonal entries either one or
zero. $A$ will identify those states of a single particle that can be reached
in one jump from any given state. For $N\in\mathbb{N}$, define a matrix
$\boldsymbol{A}_{N}=(\boldsymbol{A}_{N}(\boldsymbol{x},\boldsymbol{y}%
))_{\boldsymbol{x},\boldsymbol{y}\in\mathcal{X}^{N}}$ indexed by elements of
$\mathcal{X}^{N}$ according to $\boldsymbol{A}_{N}(\boldsymbol{x}%
,\boldsymbol{y})=\alpha(x_{l},y_{l})$ if $\boldsymbol{x}$ and $\boldsymbol{y}$
differ in exactly one index $l\in\{1,\ldots,N\}$, and $\boldsymbol{A}%
_{N}(\boldsymbol{x},\boldsymbol{y})=0$ otherwise. Then $\boldsymbol{A}_{N}$
determines which states of the $N$-particle system can be reached in one jump.
Observe that $\boldsymbol{A}_{N}$ is symmetric and irreducible with values in
$\{0,1\}$. There are many ways one can define a rate matrix $\boldsymbol{\Psi
}^{N}$ such that the corresponding Markov process is reversible with respect
to $\boldsymbol{\pi}_{N}$. Three standard ones are as follows. Let $a^{+}%
=\max\{a,0\}$. For $\boldsymbol{x},\boldsymbol{y}\in\mathcal{X}^{N}$,
$\boldsymbol{x}\neq\boldsymbol{y}$, set either
\begin{subequations}
\label{ExPrelimitGenerator}%
\begin{align}
&  & \boldsymbol{\Psi}^{N}(\boldsymbol{x},\boldsymbol{y})  &  \doteq
e^{-\left(  U_{N}(\boldsymbol{y})-U_{N}(\boldsymbol{x})\right)  ^{+}%
}\boldsymbol{A}_{N}(\boldsymbol{x},\boldsymbol{y})\label{ExPrelimitGen1}\\
&  \text{or} & \boldsymbol{\Psi}^{N}(\boldsymbol{x},\boldsymbol{y})  &
\doteq\left(  1+e^{U_{N}(\boldsymbol{y})-U_{N}(\boldsymbol{x})}\right)
^{-1}\boldsymbol{A}_{N}(\boldsymbol{x},\boldsymbol{y})\label{ExPrelimitGen2}\\
&  \text{or} & \boldsymbol{\Psi}^{N}(\boldsymbol{x},\boldsymbol{y})  &
\doteq\frac{1}{2}\left(  1+e^{-\left(  U_{N}(\boldsymbol{y})-U_{N}%
(\boldsymbol{x})\right)  }\right)  \boldsymbol{A}_{N}(\boldsymbol{x}%
,\boldsymbol{y}). \label{ExPrelimitGen3}%
\end{align}
In all three cases set $\boldsymbol{\Psi}^{N}(\boldsymbol{x},\boldsymbol{x}%
)\doteq-\sum_{\boldsymbol{y}:\boldsymbol{y}\neq\boldsymbol{x}}\boldsymbol{\Psi}%
^{N}(\boldsymbol{x},\boldsymbol{y})$, $\boldsymbol{x}\in\mathcal{X}^{N}$. The
model defined by (\ref{ExPrelimitGen1}) is sometimes referred to as
\emph{Metropolis dynamics}, and (\ref{ExPrelimitGen2}) as \emph{heat bath
dynamics} \cite{martinelli99}. The matrix $\boldsymbol{\Psi}^{N}$ is the
generator of an irreducible continuous-time finite-state Markov process with
state space $\mathcal{X}^{N}$. In what follows we will consider only
(\ref{ExPrelimitGen1}), the analysis for the other dynamics being completely analogous.

Define $H:\mathcal{X}\times\mathbb{R}^{d}\rightarrow\mathbb{R}$ by
\end{subequations}
\begin{equation}
H(x,p)=H^{x}(p)\doteq K^{x}(p)+\sum_{z\in\mathcal{X}}\left(  \frac{\partial
}{\partial p_{x}}K^{z}(p)\right)  p_{z} \label{eq:dofofH}%
\end{equation}
and $\Psi:\mathcal{X}\times\mathcal{X}\times\mathbb{R}^{d}\rightarrow
\mathbb{R}$ by%
\[
\Psi(x,y,p)\doteq H^{y}(p)-H^{x}(p),\;(x,y,p)\in\mathcal{X}\times
\mathcal{X}\times\mathbb{R}^{d}.
\]

The following lemma shows that each $\boldsymbol{\Psi}^{N}$ in
\eqref{ExPrelimitGenerator} is the infinitesimal generator of a family of
weakly interacting Markov processes in the sense of Section~\ref{subs-model}.
For example, with the dynamics (\ref{ExPrelimitGen1}) it will follow from
Lemma \ref{lem:isweakinter} that $\Gamma_{x,y}^{N}(r)\rightarrow
e^{-(\Psi(x,y,r))^{+}}\alpha(x,y)$ as $N\rightarrow\infty$.

\begin{lemma}
\label{lem:isweakinter} There exists $C<\infty$ and for each $N\in\mathbb{N}$
a function $B^{N}:\mathcal{X}\times\mathcal{X}\times\mathcal{P}(\mathcal{X}%
)\rightarrow\mathbb{R}$ satisfying
\begin{equation}
\sup_{(x,y,p)\in\mathcal{X}\times\mathcal{X}\times\mathcal{P}(\mathcal{X}%
)}|B^{N}(x,y,p)|\leq\frac{C}{N} \label{eq:equnifbd}%
\end{equation}
such that the following holds. Let $\boldsymbol{x},\boldsymbol{y}%
\in\mathcal{X}^{N}$ be such that $\boldsymbol{A}_{N}(\boldsymbol{x}%
,\boldsymbol{y})=1$, and let $l\in\{1,\ldots,N\}$ be the unique index such
that $x_{l}\neq y_{l}$. Then%
\[
U_{N}(\boldsymbol{y})-U_{N}(\boldsymbol{x})=\Psi(x_{l},y_{l},r^{N}%
(\boldsymbol{x}))+B^{N}(x_{l},y_{l},r^{N}(\boldsymbol{x})).
\]

\end{lemma}

\begin{proof}
Using the definition of $U_{N}$ we have
\begin{align}
U_{N}(\boldsymbol{y})-U_{N}(\boldsymbol{x})  &  =\sum_{i=1}^{N}K^{y_{i}}%
(r^{N}(\boldsymbol{y}))-\sum_{i=1}^{N}K^{x_{i}}(r^{N}(\boldsymbol{x}%
))\nonumber\\
&  =\sum_{i=1,i\neq l}^{N}\left(  K^{y_{i}}\left(  r^{N}(\boldsymbol{x}%
)+\frac{1}{N}(e_{y_{l}}-e_{x_{l}})\right)  -K^{x_{i}}(r^{N}(\boldsymbol{x}%
))\right) \nonumber\\
&  \quad+K^{y_{l}}\left(  r^{N}(\boldsymbol{x})+\frac{1}{N}(e_{y_{l}}%
-e_{x_{l}})\right)  -K^{x_{l}}(r^{N}(\boldsymbol{x})). \label{eq:taylor}%
\end{align}
Let $\Vert p\Vert\doteq\sum_{x}|p_{x}|$ for $p\in\mathbb{R}^{d}$. From the
$C^{2}$ property of $K$ it follows that there are $A:\mathcal{X}%
\times\mathbb{R}^{d}\times\mathbb{R}^{d}\rightarrow\mathbb{R}$ and $c_{1}%
\in(0,\infty)$ such that for all $p,q\in\mathbb{R}^{d}$, $y\in\mathcal{X}$,
\[
K^{y}(q)-K^{y}(p)=\nabla_{p}K^{y}(p)\cdot(q-p)+A(y,p,q),
\]
and
\begin{equation}
\sup_{y\in\mathcal{X},\Vert p\Vert\leq2,\Vert q\Vert\leq2}|A(y,p,q)|\leq
c_{1}\Vert p-q\Vert^{2}. \label{eq:bound1}%
\end{equation}
Also note that for some $c_{2}\in(0,\infty)$
\begin{equation}
\sup_{y\in\mathcal{X},\Vert p\Vert\leq2,\Vert q\Vert\leq2}|K^{y}%
(q)-K^{y}(p)|\leq c_{2}\Vert p-q\Vert, \label{eq:bound2}%
\end{equation}
and since $r_{z}^{N}(\boldsymbol{x})$ is the empirical measure $\frac{1}%
{N}\sum_{i=1}^{N}1_{\left\{  x_{i}=z\right\}  }$,%
\[
\sum_{z\in\mathcal{X}}\left(  \frac{\partial}{\partial p_{y_{l}}}K^{z}%
(r^{N}(\boldsymbol{x}))\right)  r_{z}^{N}(\boldsymbol{x})=\frac{1}{N}%
\sum_{i=1}^{N}\left(  \frac{\partial}{\partial p_{y_{l}}}K^{x_{i}}%
(r^{N}(\boldsymbol{x}))\right)  .
\]
Using the various definitions and in particular (\ref{eq:dofofH}) and
(\ref{eq:taylor}) we have
\[
U_{N}(\boldsymbol{y})-U_{N}(\boldsymbol{x})-\Psi(x_{l},y_{l},r^{N}%
(\boldsymbol{x}))=B^{N}(x_{l},y_{l},r^{N}(\boldsymbol{x})),
\]
where for $(x,y,p)\in\mathcal{X}\times\mathcal{X}\times\mathcal{S}$
\begin{align*}
B^{N}(x,y,p)  &  =N\sum_{z\in\mathcal{X}}A\left(  z,p,p+\frac{1}{N}%
(e_{y}-e_{x})\right)  p_{z}-A\left(  x,p,p+\frac{1}{N}(e_{y}-e_{x})\right) \\
&  \quad-\frac{1}{N}\nabla_{p}K^{x}(p)\cdot(e_{y}-e_{x})-K^{y}(p)+K^{y}\left(
p+\frac{1}{N}(e_{y}-e_{x})p\right)  .
\end{align*}
Using the bounds (\ref{eq:bound1}) and (\ref{eq:bound2}), we have that
\eqref{eq:equnifbd} is satisfied for a suitable $C<\infty$.
\end{proof}

\medskip From Lemma \ref{lem:isweakinter} we have that the jump rates of the
Markov process governed by $\boldsymbol{\Psi}^{N}$ in each of the three cases
in \eqref{ExPrelimitGenerator} depend on the components $x_{j}$, $j\neq l$,
only through the empirical measure $r^{N}(\boldsymbol{x})$. For example, with
$\boldsymbol{\Psi}^{N}$ as in \eqref{ExPrelimitGen1}, for $\boldsymbol{x}%
,\boldsymbol{y}\in\mathcal{X}^{N}$ such that $x_{l}\neq y_{l}$ for some
$l\in\{1,\ldots,N\}$, $x_{j}=y_{j}$ for $j\neq l$,
\[
\boldsymbol{\Psi}^{N}(\boldsymbol{x},\boldsymbol{y})\doteq e^{-\left(
\Psi(x_{l},y_{l},r^{N}(\boldsymbol{x}))+B^{N}(x_{l},y_{l},r^{N}(\boldsymbol{x}%
))\right)  ^{+}}\boldsymbol{A}_{N}(\boldsymbol{x},\boldsymbol{y}).
\]
Thus $\boldsymbol{\Psi}^{N}$ as in \eqref{ExPrelimitGenerator} is the
generator of a family of weakly interacting Markov processes in the sense of
Section~\ref{SectInterMarkov}. Indeed for \eqref{ExPrelimitGen1}, in the
notation of that section, $\boldsymbol{\Psi}^{N}$ is defined in terms of the
family of matrices $\{\Gamma^{N}(r)\}_{r\in\mathcal{P}(\mathcal{X})}$, where
for $x,y\in\mathcal{X}$, $x\neq y$,
\begin{equation}
\Gamma_{x,y}^{N}(r)=e^{-\left(  \Psi(x,y,r)+B^{N}(x,y,r))\right)  ^{+}}%
\alpha(x,y). \label{eq:eqprelim}%
\end{equation}

The rate matrix $\boldsymbol{\Psi}^{N}$ in \eqref{ExPrelimitGenerator} has
$\boldsymbol{\pi}_{N}$ defined in \eqref{eq:gibbsmzr} as its stationary
distribution. To see this, let $\boldsymbol{x},\boldsymbol{y}\in
\mathcal{X}^{N}$. By symmetry, $\boldsymbol{A}_{N}(\boldsymbol{x}%
,\boldsymbol{y})=\boldsymbol{A}_{N}(\boldsymbol{y},\boldsymbol{x})$. Taking
into account \eqref{eq:gibbsmzr}, it is easy to see that for any of the three
choices of $\boldsymbol{\Psi}^{N}$ according to \eqref{ExPrelimitGenerator} we
have $\boldsymbol{\pi}_{N}(\boldsymbol{x})\boldsymbol{\Psi}^{N}(\boldsymbol{x}%
,\boldsymbol{y})=\boldsymbol{\pi}_{N}(\boldsymbol{y})\boldsymbol{\Psi}%
^{N}(\boldsymbol{y},\boldsymbol{x})$. Thus $\boldsymbol{\Psi}^{N}$ satisfies
the detailed balance condition with respect to $\boldsymbol{\pi}_{N}$, and
since $\boldsymbol{\Psi}^{N}$ is irreducible, $\boldsymbol{\pi}_{N}$ is its
unique stationary distribution.

Hence by \eqref{eq:equnifbd}, the family $\{\Gamma^{N}(r)\}_{r\in
\mathcal{P}(\mathcal{X})}$ defined by \eqref{eq:eqprelim} satisfies Condition
\ref{ass-llnbasic} with
\begin{equation}
\Gamma_{x,y}(r)=e^{-\left(  \Psi(x,y,r)\right)  ^{+}}\alpha(x,y),\;x\neq
y,\;r\in\mathcal{P}(\mathcal{X}). \label{eq:eqlimgen}%
\end{equation}
With $\boldsymbol{X}^{N}$ and $\mu^{N}$ associated with $\Gamma^{N}(\cdot)$ as
in Section \ref{subs-model}, Theorem \ref{ThLLN} implies the sequence
$\{\mu^{N}\}_{N\in\mathbb{N}}$ of $D([0,\infty),\mathcal{P}(\mathcal{X}%
))$-valued random variables satisfies a law of large numbers with limit
determined by \eqref{EqLimitKolmogorov}, and with $\Gamma(\cdot)$ as in
\eqref{eq:eqlimgen}. More precisely, if $\mu^{N}(0)$ converges in distribution
to $q\in\mathcal{P}(\mathcal{X})$ as $N$ goes to infinity then $\mu^{N}%
(\cdot)$ converges in distribution to the solution $p(\cdot)$ of
\eqref{eq:eqlimgen} with $p(0)=q$. Thus $\Gamma(\cdot)$ describes the limit
model for the families of weakly interacting Markov processes of Gibbs type
introduced above. If $p\in\mathcal{P}(\mathcal{X})$ is fixed then $\Gamma(p)$
is the generator of an ergodic finite-state Markov process, and the unique
invariant distribution on $\mathcal{X}$ is given by $\pi(p)$ with
\begin{equation}
\pi(p)_{x}\doteq\frac{1}{Z(p)}\exp\left(  -H^{x}(p)\right)  ,
\label{ExLimitStationary}%
\end{equation}
where%
\[
Z(p)\doteq\sum_{x\in\mathcal{X}}\exp\left(  -H^{x}(p)\right)  .
\]

\subsection{Limit of relative entropies}

\label{SectGibbsRELimit} We will now evaluate the limit in
\eqref{eq:firstlyap} for the family of interacting $N$-particle systems
introduced in Section \ref{SectGibbsSystems}. As noted earlier, the paper
\cite{BDFR-Examples} will study the Lyapunov function properties of the limit.

\begin{theorem}
\label{ThRELimit} For $N\in\mathbb{N}$, define $\tilde{F}_{N}\!:\mathcal{P}%
(\mathcal{X})\rightarrow\lbrack0,\infty]$ by
\begin{equation}
\tilde{F}_{N}(p)\doteq\frac{1}{N}R\left(  \left.  \otimes^{N}p\right\Vert
\boldsymbol{\pi}_{N}\right)  . \label{ExPrelimitREFnct}%
\end{equation}
Then there is a constant $C\in\mathbb{R}$ such that for all $p\in
\mathcal{P}(\mathcal{X})$,
\begin{equation}
\lim_{N\rightarrow\infty}\tilde{F}_{N}(p)=\sum_{x\in\mathcal{X}}(K^{x}(p)+\log
p_{x})p_{x}-C. \label{eq:eqlyapfn}%
\end{equation}

\end{theorem}

\begin{proof}
Let $p\in\mathcal{P}(\mathcal{X})$. By the definition of relative entropy in
(\ref{eq:relent}), (\ref{ExPrelimitREFnct}), (\ref{eq:gibbsmzr}) and
(\ref{eq:norm}),%
\begin{align*}
\tilde{F}_{N}(p)  &  =\frac{1}{N}\sum_{\boldsymbol{x}\in\mathcal{X}^{N}%
}\left(  \prod_{i=1}^{N}p_{x_{i}}\right)  \log\left(  \frac{\prod_{i=1}%
^{N}p_{x_{i}}}{\boldsymbol{\pi}_{N}(\boldsymbol{x})}\right) \\
&  =\frac{1}{N}\sum_{\boldsymbol{x}\in\mathcal{X}^{N}}\left(  \prod_{i=1}%
^{N}p_{x_{i}}\right)  \left(  \sum_{i=1}^{N}\log p_{x_{i}}\right)  +\frac
{1}{N}\log Z_{N}\\
&  \mbox{}\quad+\frac{1}{N}\sum_{\boldsymbol{x}\in\mathcal{X}^{N}}\left(
\prod_{i=1}^{N}p_{x_{i}}\right)  \left(  \sum_{i=1}^{N}K(x_{i},r^{N}%
(\boldsymbol{x}))\right)  .
\end{align*}
Let $\{X_{i}\}_{i\in\mathbb{N}}$ be a sequence of i.i.d.\ $\mathcal{X}$-valued
random variables with common distribution $p$ defined on some probability
space. Then%
\begin{equation}
\frac{1}{N}\sum_{\boldsymbol{x}\in\mathcal{X}^{N}}\left(  \prod_{i=1}%
^{N}p_{x_{i}}\right)  \left(  \sum_{i=1}^{N}\log p_{x_{i}}\right)  =E\left[
\frac{1}{N}\sum_{i=1}^{N}\log  p_{X_{i}}  \right]  =E\left[
\log  p_{X_{1}}  \right]  , \label{eq:le}%
\end{equation}
and%
\begin{align*}
\frac{1}{N}\sum_{\boldsymbol{x}\in\mathcal{X}^{N}}\left(  \prod_{i=1}%
^{N}p_{x_{i}}\right)  \sum_{j=1}^{N}K(x_{j},r^{N}(\boldsymbol{x}))  &
=E\left[  \frac{1}{N}\sum_{j=1}^{N}K(X_{j},r^{N}(X_{1},\ldots,X_{N}))\right]
\\
&  =E\left[  K(X_{1},r^{N}(X_{1},\ldots,X_{N}))\right]  ,
\end{align*}
which converges to $E\left[  K(X_{1},p)\right]  $ as $N\rightarrow\infty$ due
to the strong law of large numbers and continuity of $K$.

In order to compute the limit of $\frac{1}{N}\log Z_{N}$, define a bounded
and continuous mapping $\Phi\!:\mathcal{P}(\mathcal{X})\rightarrow\mathbb{R}$
by
\[
\Phi(q)\doteq\sum_{x\in\mathcal{X}}K(x,q)q_{x}.
\]
Let $\{Y_{i}\}_{i\in\mathbb{N}}$ be i.i.d. $\mathcal{X}$-valued random
variables with common distribution $\nu$ given by $\nu_{x}\doteq\frac
{1}{|\mathcal{X}|}$, $x\in\mathcal{X}$. Then again using that $r^{N}%
(\boldsymbol{x})$ is the empirical measure of $\boldsymbol{x}$,
\begin{align*}
Z_{N}  &  =\sum_{\boldsymbol{x}\in\mathcal{X}^{N}}\exp\left(  -\sum_{i=1}%
^{N}K(x_{i},r^{N}(\boldsymbol{x}))\right) \\
&  =|\mathcal{X}|^{N}E\left[  \exp\left(  -\sum_{i=1}^{N}K(Y_{i},r^{N}%
(Y_{1},\ldots,Y_{N}))\right)  \right] \\
&  =|\mathcal{X}|^{N}E\left[  \exp\left(  -N\Phi(r^{N}(Y_{1},\ldots
,Y_{N}))\right)  \right]  .
\end{align*}
Thus by Sanov's theorem and Varadhan's theorem on the asymptotic evaluation of
exponential integrals \cite{DupEllBook}, it follows that
\[
\lim_{N\rightarrow\infty}\frac{1}{N}\log Z_{N}=-\inf_{q\in\mathcal{P}%
(\mathcal{X})}\left\{  R(q\Vert\nu)+\Phi(q)\right\}  +\log|\mathcal{X}%
|\doteq-C.
\]
Note that $C$ is finite and does not depend on $p$.

Recalling that $X_{1}$ is a random variable with distribution $p$, we have on
combining these observations that%
\begin{align*}
\lim_{N\rightarrow\infty}\tilde{F}_{N}(p)  &  =E\left[  \log  p_{X_{1}%
}  \right]  +E\left[  K(X_{1},p)\right]  -C\\
&  =\sum_{x\in\mathcal{X}}p_{x}\log p_{x}+\sum_{x\in\mathcal{X}}%
K(x,p)p_{x}-C.
\end{align*}
This proves \eqref{eq:eqlyapfn} and completes the proof.
\end{proof}

\medskip As an immediate consequence we get the following result for $K$ as in \eqref{eq:affinegibbs}.

\begin{corollary}
\label{cor:lyapforgibbs} Suppose that $K$ is defined by \eqref{eq:affinegibbs}
and let $\tilde{F}_{N}$ be as in \eqref{ExPrelimitREFnct}. Then
\begin{equation}
\lim_{N\rightarrow\infty}\tilde{F}_{N}(p)=\sum_{x\in\mathcal{X}}\left(
V(x)+\sum_{y\in\mathcal{X}}W(x,y)p_{y}+\log p_{x}\right)  p_{x}-C.
\label{eq:eqlyapfnG}%
\end{equation}

\end{corollary}

\begin{remark}
\label{rem:notrev} \emph{ In Section 4 of \cite{BDFR-Examples} it will be
shown that the function $F(p)$ defined by the right side of
\eqref{eq:eqlyapfn} satisfies a descent property: $\frac{d}{dt}F(p(t))\leq0$,
where $p(\cdot)$ is the solution of \eqref{EqLimitKolmogorov} with $\Gamma$ as
in \eqref{eq:eqlimgen}. Furthermore $\frac{d}{dt}F(p(t))=0$ if and only if
$p(t)$ is a fixed point of \eqref{EqLimitKolmogorov}, i.e., $p(t)=\pi(p(t))$.
One may conjecture that an analogous descent property holds for the function
$\widehat{F}$ obtained by taking limits of relative entropies computed in the
reverse order, namely for the function}
\begin{equation}
\widehat{F}(p)\doteq\lim_{N\rightarrow\infty}\frac{1}{N}R\left(
\boldsymbol{\pi}_{N}\Vert\otimes^{N}p\right)  ,\;p\in\mathcal{P}(\mathcal{X}).
\label{eq-widehatF}%
\end{equation}
\emph{However, in general, this is \emph{not true}, as the following example
illustrates. }

\emph{Consider the setting where $K$ is given by \eqref{eq:affinegibbs} with
environment potential $V\equiv0$, $\beta=1,$ and non-constant symmetric
interaction potential $W$ with $W\geq0$ and $W(x,x)=0$ for all $x\in
\mathcal{X}$. Then, by \eqref{ExLimitStationary}, the invariant distributions
are given by}
\[
\pi(p)_{x}=\frac{1}{Z(p)}\exp\left(  -2\sum_{y\in\mathcal{X}}W(x,y)p_{y}%
\right)  ,
\]
\emph{and the family of rate matrices $(\Gamma(p))_{p\in\mathcal{P}%
(\mathcal{X})}$ are of the form \eqref{eq:eqlimgen}, with $\Psi(x,y,p)\doteq
2\sum_{z\in\mathcal{X}}\left(  W(y,z)-W(x,z)\right)  p_{z}$. } \emph{Suppose
$W$ is such that there exists a unique solution $\pi^{\ast}\in\mathcal{P}%
(\mathcal{X})$ to the fixed point equation $\pi(p)=p$. Then using the same
type of calculations as those used to prove Theorem \ref{ThRELimit}, one can
check that $\widehat{F}$ is well defined and takes the form }
\[
\widehat{F}(p)=R\left(  \pi^{\ast}\Vert p\right)  +C,\;p\in\mathcal{P}%
(\mathcal{X})
\]
\emph{for some finite constant $C\in\mathbb{R}$ that depends on $\pi^{\ast}$
(but not on $p$). Thus the proposed Lyapunov function is relative entropy with
the independent variable in the second position, and the dynamics are of the
form (\ref{EqLimitKolmogorov}) for }$\Gamma$ \emph{that is not a constant.
While }$R\left(  \pi^{\ast}\Vert p\right)  $ \emph{satisfies the descent
property for constant ergodic matrices }$\Gamma$ \emph{such that} $\pi^{\ast
}\Gamma=0$\emph{, this property is not valid in any generality when }$\Gamma
$\emph{ depends on} $p$\emph{, and one can then easily construct examples for
which the function $\widehat{F}$ defined above does not enjoy the descent
property. }
\end{remark}

\section{General Weakly Interacting Systems}

\label{sec:genweak} The analysis of Section \ref{sec:gibbstype} crucially
relied on the fact that the stationary distributions for systems of Gibbs type
take an explicit form. In general, when the form of $\boldsymbol{\pi}_{N}$ is
not known, evaluation of the limit in \eqref{eq:firstlyap} becomes infeasible.
A natural approach then is to consider the function in \eqref{eq:fnq} and to
evaluate the quantity $\lim_{t\rightarrow\infty}\lim_{N\rightarrow\infty}%
F_{t}^{N}(q)$. In this section we will consider the problem of evaluating the
inner limit, i.e. $\lim_{N\rightarrow\infty}F_{t}^{N}(q)$. We will show that
this limit, denoted by $J_{t}(q)$, exists quite generally. In
\cite{BDFR-Examples} we will study properties of the candidate Lyapunov
function $\lim_{t\rightarrow\infty}J_{t}(q)$.

To argue the existence of $\lim_{N\rightarrow\infty}F_{t}^{N}(q)$ and to
identify the limit we begin with a general result.

\subsection{Relative entropy asymptotics for an exchangeable collection}

\label{sec:asymgenexc}

Let $\boldsymbol{Q}^{N}$ be an exchangeable probability measure on
$\mathcal{X}^{N}$. We next present a result that shows how to evaluate the
limit of
\[
\frac{1}{N}R(\otimes^{N}q\Vert\boldsymbol{Q}^{N})
\]
as $N\rightarrow\infty$, where $q\in\mathcal{S}$. Recall that $r^{N}%
:{\mathcal{X}}^{N}\rightarrow{\mathcal{P}}_{N}({\mathcal{X}})$ defined in
(\ref{def-rn}) returns the empirical measure of a sequence in ${\mathcal{X}%
}^{N}$.

\begin{definition}
Let $J:\mathcal{S}\rightarrow\lbrack0,\infty]$ be a lower semicontinuous
function. We say that $r^{N}$ under the probability law $\boldsymbol{Q}^{N}$
satisfies a locally uniform LDP on ${\mathcal{P}}({\mathcal{X}})$ with rate
function $J$ if, given any sequence $\{q_{N}\}_{N\in\mathbb{N}}$, $q_{N}%
\in\mathcal{P}_{N}(\mathcal{X})$, such that $q_{N}\rightarrow q\in
\mathcal{P}(\mathcal{X})$,
\[
\lim_{N\rightarrow\infty}\frac{1}{N}\log\boldsymbol{Q}^{N}\left(  \left\{
\boldsymbol{y}\in{\mathcal{X}}^{N}:r^{N}(\boldsymbol{y})=q^{N}\right\}
\right)  =-J(q).
\]

\end{definition}

The standard formulation of a LDP is stated in terms of bounds for open and
closed sets. In contrast, a locally uniform LDP (which implies the standard
LDP with the same rate function) provides approximations to the probability
that a random variable equals a single point. Under an appropriate
communication condition, such a strengthening is not surprising when random
variables take values in a lattice.

The following is the main result of this section. Together with a large
deviation result 
stated in Theorem \ref{thm:ldpthm}
below, it will be used to characterize $\lim_{N\rightarrow\infty}F_{t}^{N}(q)$.

\begin{theorem}
\label{thm:limre}Suppose that $r^{N}$ under the exchangeable probability law
$\boldsymbol{Q}^{N}$ satisfies a locally uniform LDP on ${\mathcal{P}%
}({\mathcal{X}})$ with rate function $J$. Suppose that $J(q)<\infty$ for all
$q\in{\mathcal{P}}({\mathcal{X}})$. Then for all $q\in{\mathcal{P}%
}({\mathcal{X}})$,
\[
\lim_{N\rightarrow\infty}\frac{1}{N}R\left(  \otimes^{N}q\left\Vert
\boldsymbol{Q}^{N}\right.  \right)  =J(q).
\]

\end{theorem}

\begin{proof}
We follow the convention that $x\log x$ equals $0$ when $x=0$. Fix
$q\in{\mathcal{P}}({\mathcal{X}})$ and note that relative entropy can be
decomposed as
\begin{align}
\lefteqn{\frac{1}{N}R\left(  \otimes^{N}q\left\Vert \boldsymbol{Q}^{N}\right.
\right)  }\label{eq:tpt1}\\
&  =\frac{1}{N}\sum_{\boldsymbol{y}\in\mathcal{X}^{N}}\left(  \prod
\limits_{i=1}^{N}q_{y_{i}}\right)  \log\left(  \prod\limits_{i=1}^{N}q_{y_{i}%
}\right)  -\frac{1}{N}\sum_{\boldsymbol{y}\in\mathcal{X}^{N}}\left(
\prod\limits_{i=1}^{N}q_{y_{i}}\right)  \log\boldsymbol{Q}^{N}(\boldsymbol{y}%
).\nonumber
\end{align}
Let $\{X_{i}\}_{i\in\mathbb{N}}$ be an i.i.d.\ sequence of $\mathcal{X}%
$-valued random variables with common probability distribution $q$. Then
exactly as in the proof of Theorem \ref{ThRELimit}, for each $N\in\mathbb{N}$
\begin{equation}
\frac{1}{N}\sum_{\boldsymbol{y}\in\mathcal{X}^{N}}\left(  \prod\limits_{i=1}%
^{N}q_{y_{i}}\right)  \log\left(  \prod\limits_{i=1}^{N}q_{y_{i}}\right)
=E\left[  \log q_{X_{1}}\right]  =\sum_{x\in\mathcal{X}}q_{x}\log q_{x}.
\label{eq:step1}%
\end{equation}

Next, consider the second term on the right side of \eqref{eq:tpt1}. Since
$\boldsymbol{Q}^{N}$ is exchangeable there is a function $G^{N}:\mathcal{P}%
_{N}(\mathcal{X})\rightarrow\lbrack0,1]$ such that $\boldsymbol{Q}%
^{N}(\boldsymbol{y})=G^{N}(r^{N}(\boldsymbol{y}))$ for all $\boldsymbol{y}%
\in\mathcal{X}^{N}$. Then for $r\in{\mathcal{P}}_{N}({\mathcal{X}})$ we can
write
\[
\boldsymbol{Q}^{N}(\left\{  \boldsymbol{y}\in\mathcal{X}^{N}:r^{N}%
(\boldsymbol{y})=r\right\}  )=\left\vert \left\{  \boldsymbol{y}\in
\mathcal{X}^{N}:r^{N}(\boldsymbol{y})=r\right\}  \right\vert G^{N}(r).
\]
For notational convenience, let $C^{N}(r)=\left\vert \left\{  \boldsymbol{y}%
\in{\mathcal{X}}^{N}:r^{N}(\boldsymbol{y})=r\right\}  \right\vert $,
$r\in\mathcal{P}_{N}(\mathcal{X})$. Rearranging the last expression gives
\begin{equation}
G^{N}(r)=\frac{\boldsymbol{Q}^{N}(\left\{  \boldsymbol{y}\in{\mathcal{X}}%
^{N}:r^{N}(\boldsymbol{y})=r\right\}  )}{C^{N}(r)}. \label{def-fnq}%
\end{equation}
Since $\boldsymbol{Q}^{N}(\boldsymbol{y})=G^{N}(r^{N}(\boldsymbol{y}))$,
\begin{equation}
\frac{1}{N}\sum_{\boldsymbol{y}\in\mathcal{X}^{N}}\left(  \prod\limits_{i=1}%
^{N}q_{y_{i}}\right)  \log\boldsymbol{Q}^{N}(\boldsymbol{y})=\frac{1}{N}%
\sum_{\boldsymbol{y}\in\mathcal{X}^{N}}\left(  \prod\limits_{i=1}^{N}q_{y_{i}%
}\right)  \log G^{N}(r^{N}(\boldsymbol{y})). \label{eq:two}%
\end{equation}
Let $\Theta^{N}:{\mathcal{P}}_{N}({\mathcal{X}})\rightarrow\mathbb{R}%
\cup\{-\infty\}$ be defined by $\Theta^{N}(r)\doteq\frac{1}{N}\log G^{N}(r)$.
Using the fact that $P((X_{1},X_{2},\ldots,X_{N})=(y_{1},\ldots,y_{N}%
))=\prod_{i=1}^{N}q_{y_{i}}$, we can express the term on the right-hand side
of (\ref{eq:two}) in terms of the i.i.d.\ sequence $\{X_{i}\}_{i\in\mathbb{N}%
}$:
\begin{align}
\frac{1}{N}\sum_{\boldsymbol{y}\in\mathcal{X}^{N}}\left(  \prod\limits_{i=1}%
^{N}q_{y_{i}}\right)  \log\boldsymbol{Q}^{N}(\boldsymbol{y})  &  =E\left[
\Theta^{N}(r^{N}(X_{1},\ldots,X_{N}))\right] \nonumber\\
&  =E\left[  \Theta^{N}\left(  \frac{1}{N}\sum_{i=1}^{N}\delta_{X_{i}}\right)
\right]  . \label{eq:824}%
\end{align}

Let $\Theta:\mathcal{S}\rightarrow\mathbb{R}$ be defined by $\Theta
(r)\doteq\sum_{x\in\mathcal{X}}r_{x}\log r_{x}-J(r)$. We now show that
\begin{equation}
\mbox{ if }q^{N}\rightarrow q,\;q_{N}\in\mathcal{P}_{N}({\mathcal{X}%
}),\mbox{ then }\Theta^{N}(q^{N})\rightarrow\Theta(q). \label{eq:unif-fin}%
\end{equation}
Fix $\varepsilon>0$. By the assumption that $r^{N}$ under $\boldsymbol{Q}^{N}$
satisfies a locally uniform LDP with rate function $J$, and that $J(q) <
\infty$, there exists $N_{0}<\infty$ such that for all $N\geq N_{0}$,
\begin{equation}
e^{-N(J(q)+\varepsilon)}\leq\boldsymbol{Q}^{N}(\left\{  \boldsymbol{y}%
:r^{N}(\boldsymbol{y})=q^{N}\right\}  )\leq e^{-N(J(q)-\varepsilon)}.
\label{eq:839}%
\end{equation}

Next, as in Theorem \ref{ThRELimit}, let $\nu$ denote the uniform measure on
${\mathcal{X}}$ and let $\boldsymbol{Q}_{0}^{N}=\otimes^{N}\nu^{N}$. We claim
that under $\boldsymbol{Q}_{0}^{N}$, $r^{N}$ satisfies a locally uniform LDP
with rate function
\begin{equation}
\tilde{J}(p)=\sum_{x\in{\mathcal{X}}} p_{x}\log p_{x}+\log|{\mathcal{X}%
}|,\quad p\in{\mathcal{P}}({\mathcal{X}}). \label{def-barj}%
\end{equation}
Indeed, elementary combinatorial arguments (see, for example, Lemma 2.1.9 of
\cite{DemZeiBook}) show that for every $N\in\mathbb{N}$,
\begin{equation}
(N+1)^{-|{\mathcal{X}}|}e^{-NR(q^{N}\Vert\nu)}\leq\boldsymbol{Q}_{0}%
^{N}(\left\{  \boldsymbol{y}:r^{N}(\boldsymbol{y})=q^{N}\right\}  )\leq
e^{-NR(q^{N}\Vert\nu)}. \label{rem-eq}%
\end{equation}
Since $\nu$ is the uniform measure on ${\mathcal{X}}$,
\[
R(q^{N}\Vert\nu)=\sum_{x\in{\mathcal{X}}}q_{x}^{N}\log q_{x}^{N}%
-\sum_{x\in{\mathcal{X}}}q_{x}^{N}\log\frac{1}{|{\mathcal{X}|}}=\tilde
{J}(q^{N}).
\]
The locally uniform LDP of $r^{N}$ under $\boldsymbol{Q}_{0}^{N}$ then follows
from the continuity of $\tilde{J}$ and that $\frac{1}{N}\log$
$(N+1)^{-|{\mathcal{X}}|}\rightarrow0$ as $N\rightarrow\infty$.

The relation
\[
\boldsymbol{Q}_{0}^{N}(\left\{  \boldsymbol{y}:r^{N}(\boldsymbol{y}%
)=q^{N}\right\}  )=\frac{C^{N}(q^{N})}{\left\vert \mathcal{X}\right\vert ^{N}}%
\]
implies there exists $\tilde{N}_{0}<\infty$ such that for all $N\geq\tilde
{N}_{0}$,
\[
e^{-N(\tilde{J}(q)+\varepsilon)}\leq\frac{C^{N}(q^{N})}{\left\vert
\mathcal{X}\right\vert ^{N}}\leq e^{-N(\tilde{J}(q)-\varepsilon)}.
\]
Combining the last display with \eqref{eq:839} and \eqref{def-fnq}, we
conclude that for $N\geq\max\{N_{0},\tilde{N}_{0}\}$
\[
e^{-N(J(q)+\varepsilon)}e^{N(\tilde{J}(q)-\varepsilon)}e^{-N\log\left\vert
\mathcal{X}\right\vert }\leq G^{N}(q^{N})\leq e^{-N(J(q)-\varepsilon
)}e^{N(\tilde{J}(q)+\varepsilon)}e^{-N\log\left\vert \mathcal{X}\right\vert
},
\]
and thus for such $N$, recalling that $\Theta^{N}(r)=\log G^{N}(r)$
\[
\tilde{J}(q)-J(q)-\log\left\vert \mathcal{X}\right\vert -2\varepsilon
\leq\Theta^{N}(q^{N})\leq\tilde{J}(q)-J(q)-\log\left\vert \mathcal{X}%
\right\vert +2\varepsilon.
\]
Recalling $\Theta(r)=\sum_{x\in\mathcal{X}}r_{x}\log r_{x}-J(r)$ and
(\ref{def-barj}), for all such $N$, $|\Theta^{N}(q^{N})-\Theta(q)|\leq
2\varepsilon$. This proves \eqref{eq:unif-fin}.

By the strong law of large numbers $r^{N}(X_{1},\ldots,X_{N})=\frac{1}{N}%
\sum_{i=1}^{N}\delta_{X_{i}}$ converges weakly to $q$ almost surely with
respect to $P$. Thus by (\ref{eq:unif-fin})
\begin{equation}
\lim_{N\rightarrow\infty}\Theta^{N}\left(  r^{N}(X_{1},\ldots,X_{N})\right)
=\Theta(q) \label{conv-tn}%
\end{equation}
almost surely. Using \eqref{eq:unif-fin} again, the property that
$\Theta(r)<\infty$ for $r\in\mathcal{P}({\mathcal{X}})$, and the compactness
of $\mathcal{P}({\mathcal{X}})$, it follows that
\[
\limsup_{N\rightarrow\infty}\sup_{r\in\mathcal{P}_{N}(\mathcal{X})}|\Theta
^{N}(r)|<\infty.
\]
Thus by (\ref{conv-tn}) and the bounded convergence theorem,
\[
\lim_{N\rightarrow\infty}E\left[  \Theta^{N}\left(  \frac{1}{N}\sum_{i=1}%
^{N}\delta_{X_{i}}\right)  \right]  =\Theta(q).
\]
When combined with \eqref{eq:824}, this implies
\[
\lim_{N\rightarrow\infty}\frac{1}{N}\sum_{\boldsymbol{y}\in\mathcal{X}^{N}%
}\left(  \prod\limits_{i=1}^{N}q_{y_{i}}\right)  \log\boldsymbol{Q}%
^{N}(\boldsymbol{y})=\lim_{N\rightarrow\infty}E\left[  \Theta^{N}\left(
\frac{1}{N}\sum_{i=1}^{N}\delta_{X_{i}}\right)  \right]  =\Theta(q).
\]
Recalling $\Theta(q)=\sum_{x\in\mathcal{X}}q_{x}\log q_{x}-J(q)$ and using
\eqref{eq:tpt1}--\eqref{eq:step1}, the last display implies
\[
\frac{1}{N}R\left(  \otimes^{N}q\left\Vert \boldsymbol{Q}^{N}\right.  \right)
\rightarrow J(q)
\]
and completes the proof.
\end{proof}

\subsection{Evaluation of the limit, Freidlin-Wentzell quasipotential, and
metastability}

We saw in Theorem \ref{thm:limre} that the limit of the relative entropies
$\frac{1}{N}R(\otimes^{N}q\Vert\boldsymbol{Q}^{N})$ is just the rate function
$J$ of the empirical measure under $\boldsymbol{Q}^{N}$, evaluated at the
marginal of the initial product distribution $\otimes^{N}q$. We next state a
condition and a theorem that imply the LDP holds for the empirical measure
$\mu^{N}(t)$, $t\ge0$, introduced in Section \ref{subs-jmp}.

\begin{condition}
\label{ass-ergodic} Suppose that for each $r\in{\mathcal{S}}$, $\Gamma
(r)=\{\Gamma_{xy}(r),x,y\in{\mathcal{X}}\}$, is the transition rate matrix of
an ergodic ${\mathcal{X}}$-valued Markov chain.
\end{condition}

We will use the following locally uniform LDP for the empirical measure process. The LDP has been established in
\cite{Leo, BorSun} while the locally uniform version used here is  taken from
\cite{DupRamWu12}.

\begin{theorem}
\label{thm:ldpthm} Assume Conditions \ref{ass-llnbasic} and \ref{ass-ergodic}.
For $t\in\lbrack0,\infty)$ let $\boldsymbol{p}^{N}(t)$ be the distribution of
${\boldsymbol{X}}^{N}(t)=(X^{1,N}(t),\ldots,X^{N,N}(t))$, where
${\boldsymbol{X}}^{N}$ is the ${\mathcal{X}}^{N}$-valued Markov process from
Section \ref{subs-model} with exchangeable initial distribution
$\boldsymbol{p}^{N}(0)$. Recall the mapping $r^{N}:\mathcal{X}^{N}%
\rightarrow\mathcal{P}_{N}(\mathcal{X})$ given by (\ref{def-rn}), i.e.,
$r^{N}(\boldsymbol{x})$ is the empirical measure of $\boldsymbol{x}$ . Assume
that $r^{N}$ under the distribution $\boldsymbol{p}^{N}(0)$ satisfies a LDP
with a rate function $J_{0}$. Then for each $t\in\lbrack0,\infty)$, $r^{N}$
under the distribution $\boldsymbol{p}^{N}(t)$ satisfies a locally uniform LDP
on ${\mathcal{P}}({\mathcal{X}})$ with a rate function $J_{t}$. Furthermore,
$J_{t}(q)<\infty$ for all $q\in{\mathcal{P}}({\mathcal{X}})$.
\end{theorem}

The rate function $J_{t}$ takes the form (see \cite{Leo, BorSun,DupRamWu12}), 
\begin{equation}
J_{t}(q)=\inf\left\{  J_{0}(\phi(0))+\int_{0}^{t}L(\phi(s),\dot{\phi
}(s))ds:\phi(t)=q\right\}  , \label{eq:JT}%
\end{equation}
where the infimum is over all absolutely continuous $\phi:[0,t]\rightarrow
\mathcal{S}$ and $L$ takes an explicit form. As an immediate consequence of
Theorems \ref{thm:limre} and \ref{thm:ldpthm}, we have the following
characterization of $\lim_{N\rightarrow\infty}F_{t}^{N}(q)$.

\begin{theorem}
\label{cor:limfnq}Assume all the conditions of Theorem \ref{thm:ldpthm}. For
$N\in\mathbb{N}$ and $t\in\lbrack0,\infty)$, let $F_{t}^{N}$ be defined as in
\eqref{eq:fnq} and $J_{t}$ be as in Theorem \ref{thm:ldpthm}. Then
\[
\lim_{N\rightarrow\infty}F_{t}^{N}(q)=J_{t}(q),\quad q\in\mathcal{P}%
(\mathcal{X}).
\]

\end{theorem}

\begin{proof}
Recall that $\boldsymbol{p}^{N}(0)$ is an exchangeable distribution on
$\mathcal{X}^{N}$, which implies that $\boldsymbol{p}^{N}(t)$ is exchangeable
for all $t\geq0$. From Theorem \ref{thm:ldpthm} $r^{N}$ under $\boldsymbol{p}%
^{N}(t)$ satisfies a locally uniform LDP with rate function $J_{t}$ such that
$J_{t}(q)$ is finite for all $q\in\mathcal{P}(\mathcal{X})$. The result now
follows from Theorem \ref{thm:limre}.
\end{proof}

\label{sec:quasandmeta}

\medskip Recall that the ideal candidate Lyapunov function based on the
descent property of Markov processes would be
\[
\lim_{N\rightarrow\infty}\lim_{t\rightarrow\infty}\frac{1}{N}R(\otimes
^{N}q\Vert\boldsymbol{p}^{N}(t)),
\]
where by ergodicity the limit is independent of $\boldsymbol{p}^{N}(0)$. If
this is not possible, another candidate is found by interchanging the order of
the limits. In this case we can apply Theorem \ref{cor:limfnq}, and then send
$t\rightarrow\infty$ to evaluate the inverted limit. Note that in general,
this limit will depend on $\boldsymbol{p}^{N}(0)$ through $J_{0}$. We will use
this limit, and in particular the form (\ref{eq:JT}), in two ways. The first
is to derive an analytic characterization for the limit as $t\rightarrow
\infty$ of $J_{t}(q)$. This characterization will be used in
\cite{BDFR-Examples}, together with insights into the structure of candidate
Lyapunov functions obtained from the Gibbs models of Section
\ref{sec:gibbstype}, to identify and verify that candidate Lyapunov functions
for various classes of models actually are Lyapunov functions. The second use
is to directly connect these limits of relative entropies with the
Freidlin-Wentzell quasipotential related to the processes $\left\{  \mu
^{N}\right\}  $. The quasipotential provides another approach to the
construction of candidate Lyapunov functions, but one based on notions of
\textquotedblleft energy conservation\textquotedblright\ and related
variational methods, and with no a priori connection with the descent property
of relative entropies for linear Markov processes. In the rest of this section
we further compare these approaches.

Suppose that $\pi^{\ast}$ is a (locally) stable equilibrium point for
$p^{\prime}=p\Gamma(p)$, so that for some relatively open subset
$\mathbb{D}\subset\mathcal{S}$ with $\pi^{\ast}\in\mathbb{D}$ and if
$p(0)\in\mathbb{D}$ then the solution to $p^{\prime}=p\Gamma(p)$ satisfies
$p(t)\rightarrow\pi^{\ast}$ as $t\rightarrow\infty$. From \cite{Leo, BorSun, DupRamWu12} it
follows that if a deterministic sequence $\mu^{N}(0)$ converges to
$p(0)\in\mathcal{S}$, then for each $T\in(0,\infty)$ $\{\mu^{N}(t)\}_{0\leq
t\leq T}$ satisfies a LDP in $D([0,T]:\mathcal{S})$ with the rate function
\[
\int_{0}^{T}L(\phi(s),\dot{\phi}(s))ds
\]
if $\phi(\cdot)$ is absolutely continuous with $\phi(0)=p(0)$, and equal to
$\infty$ otherwise. The Freidlin-Wentzell quasipotential associated with the
large time, large $N$ behavior of $\{\mu^{N}(t)\}$ and with respect to the
initial condition $\pi^{\ast}$ is defined by
\[
V^{\pi^{\ast}}(q)=\inf\left\{  \int_{0}^{T}L(\phi(s),\dot{\phi}(s))ds:\phi
(0)=\pi^{\ast},\phi(T)=q,T\in(0,\infty)\right\}
\]
where the infimum is over all absolutely continuous $\phi:[0,T]\rightarrow
\mathcal{S}$.

Next suppose $J_{0}$ is a rate function that is consistent with the weak
convergence of $r^{N}$ under $\boldsymbol{p}^{N}(0)$ to $\pi^{\ast}$ as
$N\rightarrow\infty$. One example is $\bar{J}_{0}(r)=R(r\Vert\pi^{\ast})$,
which corresponds to $\boldsymbol{p}^{N}(0)$ equal to product measure with
marginals all equal to $\pi^{\ast}$. A second example is $J_{0}^{\pi^{\ast}%
}(r)=0$ when $r=\pi^{\ast}$ and $\infty$ otherwise, which corresponds to a
\textquotedblleft nearly deterministic\textquotedblright\ initial condition
$\boldsymbol{p}^{N}(0)$. To simplify we will consider just $J_{0}^{\pi^{\ast}%
}$. \ All other choices of $J_{0}$ bound $J_{0}^{\pi^{\ast}}$ from below and,
while leading to other candidate Lyapunov functions, they will also bound the
one corresponding to $J_{0}^{\pi^{\ast}}$ from below. Using the fact that
\[
t\mapsto J_{t}^{\pi^{\ast}}(q)\doteq\inf\left\{  J_{0}^{\pi^{\ast}}%
(\phi(0))+\int_{0}^{t}L(\phi(s),\dot{\phi}(s))ds:\phi(t)=q\right\}
\]
is monotonically decreasing, it follows that
\[
F^{\pi^{\ast}}(q)\doteq\lim_{t\rightarrow\infty}J_{t}^{\pi^{\ast}}%
(q)=V^{\pi^{\ast}}(q).
\]

Thus for a particular choice of $J_{0}$ these two different perspectives lead
to the same candidate Lyapunov function. However, this connection does not
seem a priori obvious, and we note the following distinctions. For example,
the distributions involved in the construction of $F^{\pi^{\ast}}$ via limits
of relative entropy are the product measures $\otimes_{1}^{N}q$ on
$\mathcal{X}^{N}$ with marginal $q$, and an a priori unrelated distribution
$\boldsymbol{p}^{N}(t)$ which is the joint distribution of $N$ particles at
time $t$. In contrast, the distribution relevant in the construction via the
quasipotential is the measure induced on path space by $\{\mu^{N}%
(\cdot)\}_{N\in\mathbb{N}}$ and with a sequence of initial conditions $\mu^n(0)$ which converge super-exponentially fast to $\pi^{\ast}$, and
$V^{\pi^{\ast}}(q)$ \ is defined in terms of a sample path rate function for
$\{\mu^{N}(\cdot)\}_{N\in\mathbb{N}}$ constrained to hit $q$ at the terminal time.

For both of these approaches there is a need to consider a large time limit.
When using relative entropy, to guarantee a monotone nonincreasing property
both distributions appearing in the relative entropy must be advanced by the
same amount of time. Hence it will serve as a Lyapunov function for all $q$
only if $\boldsymbol{p}^{N}(t)$ is essentially independent of $t$, which
requires sending $t\rightarrow\infty$. When using a variational formulation to
define a Lyapunov function via \textquotedblleft energy
storage\textquotedblright\ a time independent function is produced only if one
allows an arbitrarily large amount of time to go from $\pi$ to $q$, and thus
we only construct the quasipotential by allowing $T\in(0,\infty)$ in the
definition of $V^{\pi^{\ast}}$.

It is also interesting to ask what is lost by inverting the limits on $t$ and
$N$. To discuss this point we return to a particular model described in
Section \ref{sec:gibbstype}. Let $V:{\mathcal{X}}\rightarrow\mathbb{R}$,
$W:{\mathcal{X}}\times{\mathcal{X}}\rightarrow\mathbb{R}$ be given functions,
$\beta>0$, and associate interacting particle systems as in Section
\ref{sec:gibbstype}. Recall that for this family of models $F(q)$, as
introduced in (\ref{eq:firstlyap}), is given as
\begin{equation}
F(q)=\sum_{x\in\mathcal{X}}q_{x}\log q_{x}+\sum_{x\in\mathcal{X}}%
V(x)q_{x}+\beta\sum_{x,y\in\mathcal{X}}W(x,y)q_{x}q_{y}. \label{gibbslyapfn}%
\end{equation}
It is easy to check that $F$ is $C^{1}$ on $\mathcal{S}^{\circ}$. One can show
that in general multiple fixed points of the forward equation
\eqref{EqLimitKolmogorov} exist and the function $F$ serves as a local
Lyapunov function at all those fixed points which correspond to local minima.
In contrast, local Lyapunov functions constructed by taking the limits in the
order $N\rightarrow\infty$ and then $t\rightarrow\infty$ lose all
metastability information, and hence serve as local Lyapunov functions for the
point $\pi^{\ast}$ used in their definition.

\end{document}